\documentclass{amsart}
\usepackage{amssymb}
\usepackage{mathtools}
\usepackage{pdfsync}
\usepackage[pdftex,pagebackref,colorlinks=true,urlcolor=blue,linkcolor=blue,citecolor=red]{hyperref}
\usepackage[capitalize]{cleveref}
\usepackage{xcolor}
\usepackage[normalem]{ulem}

\DeclareMathOperator{\Widim}{Widim}
\DeclareMathOperator{\mdim}{mdim}
\DeclareMathOperator{\ord}{ord}
\DeclareMathOperator{\diam}{diam}
\DeclareMathOperator{\mesh}{mesh}
\DeclareMathOperator{\vol}{vol}

\newtheorem{theorem}{Theorem}[section]
\newtheorem{lemma}[theorem]{Lemma}
\newtheorem{proposition}[theorem]{Proposition}

\theoremstyle{definition}

\theoremstyle{remark}

\numberwithin{equation}{section}

\begin{document}

\title[Directional mean dimension]{Directional mean dimension and continuum-wise expansive $\mathbb{Z}^k$-actions}

\title[Directional mean dimension]{Directional mean dimension and continuum-wise expansive $\mathbb{Z}^k$-actions}

\author[Sebasti\'an Donoso]{Sebasti\'an Donoso}
\address{Center for Mathematical Modeling and Department of Mathematical Engineering, University of Chile and IRL 2807 - CNRS}
\email{sdonoso@dim.uchile.cl}
\thanks{S.D. was supported by Centro de Modelamiento Matemático (CMM), ACE\textsf{210010} and FB\textsf{210005}, BASAL funds for centers of excellence from ANID-Chile, and ANID/Fondecyt/\textsf{1200897}.}

\author[Lei Jin]{Lei Jin}
\address{Center for Mathematical Modeling, University of Chile and IRL 2807 - CNRS}
\email{jinleim@impan.pl}
\thanks{ L.J. was supported by Basal Funding AFB \textsf{170001} and ANID/Fondecyt/\textsf{3190127}, and was partially supported by NNSF of China No. \textsf{11971455}.}

\author[Alejandro Maass]{Alejandro Maass}
\address{Center for Mathematical Modeling and Department of Mathematical Engineering, University of Chile and IRL 2807 - CNRS}
\email{amaass@dim.uchile.cl}
\thanks{A.M. was supported by Centro de Modelamiento Matemático (CMM), ACE\textsf{210010} and FB\textsf{210005}, BASAL funds for centers of excellence from ANID-Chile.}

\author[Yixiao Qiao]{Yixiao Qiao}
\address{School of Mathematics and Statistics, Guangdong University of Technology, Guangzhou, 510006, China}
\email{yxqiao@gdut.edu.cn}
\thanks{Y.Q. was supported by NNSF of China No. \textsf{11901206}.}


\date{\today}


\commby{}

\begin{abstract}
We study directional mean dimension of $\mathbb{Z}^k$-actions (where $k$ is a positive integer). On the one hand, we show that there is a $\mathbb{Z}^2$-action whose directional mean dimension (considered as a $[0,+\infty]$-valued function on the torus) is not continuous. On the other hand, we prove that if a $\mathbb{Z}^k$-action is continuum-wise expansive, then the values of its $(k-1)$-dimensional directional mean dimension are bounded. This is a generalization (with a view towards Meyerovitch and Tsukamoto's theorem on mean dimension and expansive multiparameter actions) of a classical result due to Ma\~n\'e: Any compact metrizable space admitting an expansive homeomorphism (with respect to a compatible metric) is finite-dimensional.
\end{abstract}

\maketitle


\section{Introduction}
The notion of expansiveness (or unstable homeomorphisms), introduced in 1950 by Utz \cite{Utz}, is a dynamical property shared by a class of systems exhibiting chaotic behaviors. Let $(X,d)$ be a compact metric space. A homeomorphism $T\colon X\to X$ is said to be {\em expansive} if there is a constant $c>0$ satisfying that for any distinct points $x,x^\prime\in X$ there exists $n\in\mathbb{Z}$ such that $d(T^nx,T^nx^\prime)>c$. In 1979, Ma\~n\'e \cite{Mane} established a fairly surprising result which states that any compact metrizable space admitting an expansive homeomorphism (with respect to a compatible metric) is finite dimensional. In contrast to $\mathbb{Z}$-actions, there do exist expansive $\mathbb{Z}^2$-actions on infinite dimensional compact metric spaces \cite{Meyerovitch--Tsukamoto,Shi--Zhou}. Such examples seem to indicate that Ma\~n\'e's theorem cannot be extended to $\mathbb{Z}^k$-actions, for $k\ge2$ or to more general group actions. However, Meyerovitch and Tsukamoto \cite{Meyerovitch--Tsukamoto} succeeded in finding a reasonable framework to study extension of Ma\~n\'e's results, with a view towards {\em mean dimension}.

Mean dimension was introduced by Gromov \cite{Gromov} in 1999. It is a topological invariant of dynamical systems, whose advantage has been shown excellently in the study of infinite dimensional and infinite topological entropy systems. We will review its definition in Subsection \ref{subsec:mdim}. It is worth mentioning that mean dimension has deep and significant applications to dynamical systems. In particular, it is intimately connected with the \textit{embedding problem}. For related results we refer to \cite{Gutman--Lindenstrauss--Tsukamoto,Gutman--Qiao--Tsukamoto,Gutman--Tsukamoto,Lindenstrauss,Lindenstrauss--Tsukamoto1,Lindenstrauss--Tsukamoto2,Lindenstrauss--Weiss}.

Meyerovitch and Tsukamoto \cite{Meyerovitch--Tsukamoto} provided a striking generalization of Ma\~n\'e's theorem to $\mathbb{Z}^k$-actions. They showed that if $(X,\mathbb{Z}^k,T)$ is an expansive $\mathbb{Z}^k$-action (where $k$ is a positive integer) and if a $\mathbb{Z}^{k-1}$-action $(X,\mathbb{Z}^{k-1},R)$ satisfies that $R\colon \mathbb{Z}^{k-1}\times X\to X$ commutes with $T\colon \mathbb{Z}^k\times X\to X$, then $(X,\mathbb{Z}^{k-1},R)$ has finite mean dimension. Note that this statement reduces to Ma\~n\'e's result when $k=1$, since a trivial action has finite mean dimension if and only if the space is finite dimensional. Moreover, they introduced the notion of \textit{directional mean dimension} \cite{Meyerovitch--Tsukamoto}, which was initially suggested by Lind, mimicking the definition of directional entropy in \cite{Boyle--Lind}. Directional mean dimension is able to measure the ``averaged dimension'' of a dynamical system along a given subspace or direction\footnote{We shall use the word ``direction" even if the subspace is not one dimensional}. We state its precise definition (for $\mathbb{Z}^k$-actions) in Section 3. This is the starting point of what we plan to investigate in the present paper. Our motivation lies mainly in two aspects, which we explain shortly as follows.

For a $\mathbb{Z}^2$-action, let us consider its directional mean dimension as a $[0,+\infty]$-valued function on the torus. An immediate and natural problem is if such a function is \textit{continuous}. We remark that if this is true, then Meyerovitch and Tsukamoto's theorem will imply that if a $\mathbb{Z}^2$-action is expansive then its directional mean dimension along any direction is finite (because a continuous $[0,+\infty]$-valued function on a compact metric space is bounded by a finite real number provided it is finite at some point). However, it turns out that we cannot expect directional mean dimension functions to have such a strong property. We will construct a $\mathbb{Z}^2$-action whose directional mean dimension function fails to be continuous on the torus.

Let $\mathbb{S}=\{\vec{v}=(v_1,v_2)\in\mathbb{R}^2:v_1^2+v_2^2=1\}$. For a $\mathbb{Z}^2$-action $(X,\mathbb{Z}^2,T)$ we denote by $\mdim(X,T,\langle\vec{v}\rangle^{\perp})$ the directional mean dimension of $(X,\mathbb{Z}^2,T)$ with respect to the $1$-dimensional subspace of $\mathbb{R}^2$ orthogonal to a direction $\vec{v}\in\mathbb{S}$.
\begin{theorem}\label{thm:discontinuity}
For any $0\le \rho \le+\infty$ there is a $\mathbb{Z}^2$-action $(X,\mathbb{Z}^2,T)$ such that
\begin{itemize}
\item if $\vec{v}\in\{(-1,0),(1,0)\}$ then $\mdim(X,T,\langle\vec{v}\rangle^{\perp})=\rho$;
\item if $\vec{v}\in\mathbb{S}\setminus\{(-1,0),(1,0)\}$ then $\mdim(X,T,\langle\vec{v}\rangle^{\perp})=0$.
\end{itemize}
\end{theorem}
Theorem \ref{thm:discontinuity} shows in particular that directional mean dimension need \textit{not} depend continuously on directions, but still leaves open if expansive $\mathbb{Z}^2$-actions have finite directional mean dimension for all directions. We shall address this question in a wider context. Indeed, different versions of expansiveness were well explored previously. For instance, there are entropy-expansiveness \cite{Bowen}, continuum-wise expansiveness \cite{Kato}, pointwise expansiveness \cite{Reddy} and positive expansiveness \cite{Schwartzman}. Among all the variants, \textit{continuum-wise expansiveness} attracts a lot of attention. We will recall its definition formally in Subsection \ref{subsec:cwexpansive}. It is worth mentioning that the class of continuum-wise expansive homeomorphisms contains, in addition to expansive homeomorphisms, many important homeomorphisms which are \textit{not} expansive; for example, the shift maps of Knaster's indecomposable chainable continua \cite{Kato}; meanwhile, in contrast to expansive homeomorphisms, there exist continuum-wise expansive homeomorphisms on the pseudo-arc. Kato \cite{Kato} extended Ma\~n\'e's result to the setting of continuum-wise expansive homeomorphisms. A purpose of this paper is to prove that all continuum-wise expansive $\mathbb{Z}^k$-actions (in particular, all expansive $\mathbb{Z}^k$-actions) must have finite directional mean dimension.
\begin{theorem}\label{maintheorem2}
Let $k$ be a positive integer. If a $\mathbb{Z}^k$-action $(X,\mathbb{Z}^k,T)$ is continuum-wise expansive, then the $(k-1)$-dimensional directional mean dimension of $(X,\mathbb{Z}^k,T)$ with respect to any direction is finite. More precisely, they are uniformly bounded by a finite number which depends only on $(X,\mathbb{Z}^k,T)$.
\end{theorem}
We remark that the finiteness in the statement of Theorem \ref{maintheorem2} cannot be improved to zero. In fact, it is possible to construct an expansive $\mathbb{Z}^2$-action (even minimal) of positive directional mean dimension (for details see \cite{Meyerovitch--Tsukamoto}). When $k=1$, Theorem \ref{maintheorem2} coincides with Ma\~n\'e's and Kato's results.

This paper is organized as follows: In Section 2, we collect basic notions, terminologies, and some known propositions in mean dimension and continuum theory. In Section 3, we study directional mean dimension of $\mathbb{Z}^k$-actions, where we shall present its fundamental properties and where we shall provide a constructive proof of Theorem \ref{thm:discontinuity}. In Section 4, we will prove Theorem \ref{maintheorem2}.

\section{Preliminaries}\subsection{Continua}By a \textbf{continuum} we understand a connected compact metric space. We refer to \cite{Nadler2,Nadler1} for a systematic study of continuum theory. Let $(X,d)$ be a compact metric space. By the {\em hyperspace} of $X$, denoted by $2^X$, we mean the space of all nonempty closed subsets of $X$, equipped with the Hausdorff metric $d_H(A,B)=\inf\{\epsilon>0:B\subset U_\epsilon(A),A\subset U_\epsilon(B)\}$, where $U_\epsilon(A)$ denotes the $\epsilon$-neighbourhood (with respect to the metric $d$ on $X$) of $A$ in $X$. We denote by $C(X)$ the space of all nonempty connected closed subsets of $X$, endowed with the Hausdorff metric $d_H$. It is well known \cite[Theorems 4.13, 4.17]{Nadler1} that if $X$ is a compact metric space, then both $2^X$ and $C(X)$ are compact. For a topological space $X$ the {\em connected component} of a point $x\in X$ is the union of all the connected subsets of $X$ containing $x$. Every connected component is connected and closed in $X$. The connected components of the points of $X$ form a partition of $X$. When concerning a space the term {\em nondegenerate} will always mean that the space consists of at least two points.

\subsection{Continuum-wise expansive $\mathbb{Z}^k$-actions}\label{subsec:cwexpansive}Let $k$ be a positive integer. We say that a triple $(X,\mathbb{Z}^k,T)$ is a {\em $\mathbb{Z}^k$-action} if $X$ is a compact metric space and $T\colon \mathbb{Z}^k\times X\to X$, $(n,x)\mapsto T^nx$ is a continuous mapping satisfying $T^0x=x$ and $T^{m+n}(x)=T^m(T^n(x))$ for all $m,n\in\mathbb{Z}^k$ and all $x\in X$. A $\mathbb{Z}^k$-action $(X,\mathbb{Z}^k,T)$ is said to be {\em continuum-wise expansive} if there is a constant $c>0$ satisfying that for every nondegenerate $A\in C(X)$ there exists $n\in\mathbb{Z}^k$ such that $\diam T^n(A)>c$. Such a constant $c>0$ is called an {\em expansivity constant} for $(X,\mathbb{Z}^k,T)$. Note that expansive $\mathbb{Z}^k$-actions are obviously continuum-wise expansive.

\subsection{Topological dimension} Let $X$ be a compact metric space. The {\em join} of open covers $\mathcal{U}$ and $\mathcal{V}$ of $X$ is the open cover  $\mathcal{U}\vee\mathcal{V}\coloneqq \{U\cap V:U\in\mathcal{U},V\in\mathcal{V}\}$. For a cover $\mathcal{U}$ of $X$ we set $\mesh(\mathcal{U},d)=\sup_{U\in\mathcal{U}}\diam(U)$. A cover $\mathcal{V}$ {\em refines} $\mathcal{U}$ if for every $V\in\mathcal{V}$ there exists $U\in\mathcal{U}$ with $V\subset U$. Let $\mathcal{U}=\{U_i\}_{i=1}^n$ be a finite cover of $X$. We define the {\em order} of $\mathcal{U}$ by $\ord(\mathcal{U})=\max_{x\in X}\sum_{i=1}^n1_{U_i}(x)-1$ and the {\em degree} of $\mathcal{U}$ by $\mathcal{D}(\mathcal{U})=\min_\mathcal{V}\ord(\mathcal{V})$, where $\mathcal{V}$ runs over all finite open covers of $X$ refining $\mathcal{U}$. The {\em topological dimension} of $X$ is defined by $\dim(X)=\sup_\mathcal{U}\mathcal{D}(\mathcal{U})$, where $\mathcal{U}$ runs over all finite open covers of $X$. As follows we list some elementary propositions about $\mathcal{D}$, which will be used in this paper.

\begin{lemma}[{\cite[Propositions 1.6.5, 4.4.5]{Coornaert}}]\label{lem:orderclosedcover}
Let $X$ be a compact metric space and $\alpha$ a finite open cover of $X$. Then the following statements hold:\begin{enumerate}\item$\mathcal{D}(\alpha)=\min_\beta\ord(\beta)$, where $\beta$ runs over all finite closed covers of $X$ refining $\alpha$;
\item$\mathcal{D}(\alpha)\le l$ if and only if there is a compact metric space $P$ of topological dimension $l$ and a continuous mapping $f\colon X\to P$ which is $\alpha$-compatible \textup(i.e. there exists a finite open cover $\beta$ of $P$ such that $f^{-1}(\beta)$ refines $\alpha$\textup).
\end{enumerate}
\end{lemma}

\begin{lemma}[{\cite[Corollary 2.5]{Lindenstrauss--Weiss}}]\label{lem:subadditivedegree}Suppose that $\mathcal{U}$ and $\mathcal{V}$ are finite open covers of a compact metric space $X$. Then we have $\mathcal{D}(\mathcal{U}\vee\mathcal{V})\le\mathcal{D}(\mathcal{U})+\mathcal{D}(\mathcal{V})$.\end{lemma}

Let $(X,d)$ be a compact metric space, $Y$ a topological space, and $\epsilon>0$. A continuous mapping $f\colon X\to Y$ is called an {\em $\epsilon$-embedding} if $f(x)=f(x^\prime)$ implies $d(x,x^\prime)<\epsilon$ for all $x,x^\prime\in X$. We define $\Widim_\epsilon(X,d)=\min_K\dim(K)$, where $K$ ranges over all compact metric spaces admitting an $\epsilon$-embedding $f\colon X\to K$. Here we put a compatible metric $d$ on $X$ explicitly in the notation $\Widim_\epsilon(X,d)$ because $\Widim_\epsilon(X,d)$ does depend on $d$. Nevertheless, it is well known \cite[Chapter 4]{Coornaert}, for any compatible metric $d$ on $X$, that $\dim(X)=\lim_{\epsilon\to0}\Widim_\epsilon(X,d)$.\subsection{Mean dimension}\label{subsec:mdim}Let $k$ and $N$ be positive integers. We denote$$[-N,N]^k=\{(n_i)_{i=1}^k\in\mathbb{Z}^k:-N\le n_i\le N,1\le i\le k\},$$ $$\partial[-N,N]^k=\bigcup_{j=1}^k\{(x_i)_{i=1}^k\in[-N,N]^k:x_j\in\{-N,N\}\}.$$

We caution the readers that the notation $[-N,N]^k$ above depends on the context. In some places in the text, $[-N,N]$ may mean a real interval, but we believe that this will not cause any confusion at all. Let $(X,\mathbb{Z}^k,T)$ be a $\mathbb{Z}^k$-action. We set for a finite open cover $\mathcal{U}$ of $X$
\begin{equation}\label{def:mdim-degree}\mdim(X,T;\mathcal{U})=\lim_{N\to+\infty}\frac{\mathcal{D}(\bigvee_{n\in[-N,N]^k}T^n\mathcal{U})}{(2N+1)^k}
\end{equation}
and define the {\em mean dimension} of $(X,\mathbb{Z}^k,T)$ by $$\mdim(X,\mathbb{Z}^k,T)=\sup_\mathcal{U}\mdim(X,T;\mathcal{U}),$$
where $\mathcal{U}$ ranges over all finite open covers of $X$. The limit in \eqref{def:mdim-degree} always exists due to Lemma \ref{lem:subadditivedegree}. Clearly, $\mdim(X,\mathbb{Z}^k,T)\in[0,+\infty]$. An easy observation is that if $\dim(X)<+\infty$ then $\mdim(X,\mathbb{Z}^k,T)=0$. Moreover, it was shown \cite[Theorem 4.2]{Lindenstrauss--Weiss} that if $(X,\mathbb{Z}^k,T)$ has finite topological entropy, then $\mdim(X,\mathbb{Z}^k,T)=0$. Typical examples of dynamical systems having positive mean dimension are \textit{full shifts} (over a positive dimension alphabet). Let $D$ be a positive integer. The full shift on $([0,1]^D)^{\mathbb{Z}^k}$ is defined by $\sigma\colon \mathbb{Z}^k\times([0,1]^D)^{\mathbb{Z}^k}\to([0,1]^D)^{\mathbb{Z}^k}$, $(n,(x_m)_{m\in{\mathbb{Z}^k}})\mapsto(x_{m+n})_{m\in{\mathbb{Z}^k}}$. It was proved in \cite{Lindenstrauss--Weiss} that $\mdim(([0,1]^D)^{\mathbb{Z}^k},\sigma)=D$. Let $d$ be a compatible metric on $X$. For any finite subset $\Omega$ of $\mathbb{Z}^k$ we define $d_\Omega^T$ by$$d_\Omega^T(x,x^\prime)=\max_{n\in\Omega}d(T^nx,T^nx^\prime),\quad\text{for } x,x^\prime\in X.$$

It is clear that $d_\Omega^T$ is also a compatible metric on $X$. Furthermore, we can verify that $$\mdim(X,\mathbb{Z}^k,T)=\lim_{\epsilon\to0}\left(\lim_{N\to+\infty}\frac{\Widim_\epsilon(X,d_{[-N,N]^k}^T)}{(2N+1)^k}\right).$$

\section{Directional mean dimension}
Let $k\in\mathbb{N}$ and fix a real number $r>\sqrt{k}/2$.
 Let $V\subseteq \mathbb{R}^k$ be an $h$-dimensional subspace of $\mathbb{R}^k$. Define $B_r(V)\coloneqq \{u\in\mathbb{Z}^k:|u-w|<r\text{ for some }w\in V\}$. For a $\mathbb{Z}^k$-action $(X,\mathbb{Z}^k,T)$, with a compatible metric $d$ on $X$, we define the {\em $h$-dimensional directional mean dimension} of $(X,\mathbb{Z}^k,T)$ with respect to $V$ by
\begin{equation}\label{equ:definitiondirectmdim}
\mdim(X,T,V)=\lim_{\epsilon\to0}\left(\liminf_{N\to+\infty}\frac{\Widim_\epsilon(X,d_{B_r(V)\cap[-N,N]^k}^T)}{\vol_{h}(V\cap[-N,N]^k)}\right),
\end{equation}
where $\vol_{h}(V\cap[-N,N]^k)$ denotes the $h$-dimensional volume inside $V$ of its intersection with the real cube $[-N,N]^k$ \footnote{It is possible to give a much more general definition, following the notions of directional entropy in \cite{Boyle--Lind}, but for the sake of simplicity of the paper, we do not do so here.}. It follows that there exist positive constants $C_1$ and $C_2$ that depend only on the dimension $k$ and are independent of $V$ such that 
$$
C_1 N^{h} \leq \vol_{h}(V\cap[-N,N]^k)
\leq C_2 N^{h}.
$$
We note that we shall write $|A|$ for the cardinality of a set $A$ in the sequel.

The constant $r>\sqrt{k}/2$ is to guarantee that for every point $u\in V$ there is $v\in\mathbb{Z}^k$ with $|u-v|<r$, and hence $B_r(V)$ is always an infinite set. We do not include $r$ in the notation $\mdim(X,T,V)$ because the definition of directional mean dimension is independent of the choices of $r>\sqrt{k}/2$ and the compatible metrics $d$ on $X$ (see Proposition \ref{prop:ssss}).

For each $1\le i\le k$ we set $\vec{e_i}=(0,\dots,0,1,0,\dots,0)$ (i.e. having $1$ only on the $i$-th coordinate) and denote by $\mathbb{Z}_i^k=\langle e_i \rangle^{\perp}\cap \mathbb{Z}^k $ the subgroup of $\mathbb{Z}^k$ consisting of all the elements having zero on the $i$-th coordinate. For a $\mathbb{Z}^k$-action $(X,\mathbb{Z}^k,T)$ and every $1\le i\le k$ we let $(X,\mathbb{Z}^{k-1},T|_{\mathbb{Z}_i^k})$ (which is a $\mathbb{Z}^{k-1}$-action) be the restriction of $T$ to the subgroup $\mathbb{Z}_i^k$.
\begin{proposition}\label{prop:ssss}
The following statements are true for all $\mathbb{Z}^k$-actions $(X,\mathbb{Z}^k,T)$:

(1) For any $1\leq h\leq k$ and any $h$-dimensional subspace $V\subseteq \mathbb{R}^k$, the value $\mdim(X,T,V)$ is independent of $r>\sqrt{k}/2$.

(2) Directional mean dimension does not depend on a choice of a compatible metric on $X$; more precisely, for any $h$-dimensional subspace $V\subseteq \mathbb{R}^k$ and $r>\sqrt{k}/2$
\begin{equation} \label{eq:mdim_def} \mdim(X,T,V)=\sup_{\mathcal{U}}\left(\liminf_{N\to+\infty}\frac{\mathcal{D}(\vee_{n\in B_r(V)\cap[-N,N]^k}T^n\mathcal{U})}{\vol_{h}(V\cap[-N,N]^k)}\right),\end{equation} where $\mathcal{U}$ ranges over all finite open covers of $X$.

(3) For any $1\le i\le k$ we have $\mdim(X,T,\langle \vec{e_i}\rangle^{\perp})=\mdim(X,\mathbb{Z}^{k-1},T|_{\mathbb{Z}_i^k})$.

\end{proposition}
\begin{proof}
(1)  In the first part of the proof we shall use the notation $\mdim(X,T,V;r)$ for the right hand side of \eqref{eq:mdim_def} (instead of $\mdim(X,T,V)$), since our current goal is to show that this quantity does not depend on $r>\sqrt{k}/2$. Take $\sqrt{k}/2<r_1<r_2$. On the one hand, since $$d_{B_{r_1}(V) \cap[-N,N]^k}^T(x,y) \le  d_{B_{r_2}(V)\cap[-N,N]^k}^T(x,y)$$ for all $x,y\in X$ and $N\in\mathbb{N}$, we have  \[ \Widim_\epsilon(X,d_{B_{r_1}(V)\cap[-N,N]^k}^T)\le\Widim_\epsilon(X,d_{B_{r_2}(V)\cap[-N,N]^k}^T)\] for all $N\in\mathbb{N}$ and $\epsilon>0$. Thus $\mdim(X,T,V;r_1)\le\mdim(X,T,V;r_2)$. On the other hand, there exists a finite subset $F$ of $\mathbb{Z}^k$ such that for any $N\in\mathbb{N}$ the set $B_{r_2}(V)\cap[-N,N]^k$ is contained in $\bigcup_{m\in F}\left((B_{r_1}(V)\cap[-N,N]^k)+m\right)$. For a given $\epsilon>0$ there is $\delta>0$ such that $d(x,y)<\delta$ implies $d_F^T(x,y)<\epsilon$. It follows that for any $N\in\mathbb{N}$ we have $\Widim_\epsilon(X,d_{B_{r_2}(V)\cap[-N,N]^k}^T)\le\Widim_\delta(X,d_{B_{r_1}(V)\cap[-N,N]^k}^T)$. This implies that $\mdim(X,T,V;r_2)\le\mdim(X,T,V;r_1)$.

(2) To show$$\sup_\mathcal{U}\left(\liminf_{N\to+\infty}\frac{\mathcal{D}(\vee_{n\in B_r(V)\cap[-N,N]^k}T^n\mathcal{U})}{\vol_{h}(V\cap[-N,N]^k)}\right)\le\mdim(X,T,\vec{v}),$$we take a finite open cover $\mathcal{U}$ of $X$ and let $\lambda>0$ be a Lebesgue number of $\mathcal{U}$ with respect to a compatible metric $d$ on $X$. It suffices to prove for every $N\in\mathbb{N}$ that $\mathcal{D}(\vee_{n\in B_r(V)\cap[-N,N]^k}T^n\mathcal{U})\le\Widim_\lambda(X,d_{B_r(V)\cap[-N,N]^k}^T)$. In fact, we take a compact metric space $P$ with $\dim(P)=\Widim_\lambda(X,d_{B_r(V)\cap[-N,N]^k}^T)$ and a $\lambda$-embedding $f\colon X\to P$ with respect to $d_{B_r(V)\cap[-N,N]^k}^T$. Then it follows that $\diam T^n(f^{-1}(p))<\lambda$ and hence $T^n(f^{-1}(p))$ is contained in some element in $\mathcal{U}$ for every $n\in B_r(V)\cap[-N,N]^k$ and every $p\in P$. Thus, $f^{-1}(p)$ is contained in some element in $\bigvee_{n\in B_r(V)\cap[-N,N]^k}T^n\mathcal{U}$ for every $p\in P$. For every $U\in\bigvee_{n\in B_r(V)\cap[-N,N]^k}T^n\mathcal{U}$ we set $\tilde{U}=\{p\in P:f^{-1}(p)\subset U\}$. Clearly, $f^{-1}(\tilde{U})\subset U$, the sets $\tilde{U}$'s cover $P$, and each $\tilde{U}$ is open in $P$. So $f$ is $\bigvee_{n\in B_r(V)\cap[-N,N]^k}T^n\mathcal{U}$-compatible, and therefore by Lemma \ref{lem:orderclosedcover} we see that $\mathcal{D}(\vee_{n\in B_r(V)\cap[-N,N]^k}T^n\mathcal{U})\le\dim(P)$. 

To show
$$\mdim(X,T,V)\le\sup_\mathcal{U}\left(\liminf_{N\to+\infty}\frac{\mathcal{D}(\vee_{n\in B_r(V)\cap[-N,N]^k}T^n\mathcal{U})}{\vol_{h}(V\cap[-N,N]^k)}\right),$$we let $\epsilon>0$, $N\in\mathbb{N}$, and take a finite open cover $\mathcal{U}$ of $X$ with $\mesh(\mathcal{U},d)<\epsilon$. By Lemma \ref{lem:orderclosedcover}, we can find a compact metric space $P$ with $$\dim(P)=\mathcal{D}(\vee_{n\in B_r(V)\cap[-N,N]^k}T^n\mathcal{U})$$ and a continuous $f\colon X\to P$ which is $\bigvee_{n\in B_r(V)\cap[-N,N]^k}T^n\mathcal{U}$-compatible. So we have $\Widim_\epsilon(X,d_{B_r(V)\cap[-N,N]^k}^T)\le\dim(P)$, as required.

Statement (3) follows directly from definition.
\end{proof}
In what follows we are mainly interested in $(k-1)$-dimensional subspaces of $\mathbb{Z}^k$. Note that via the orthogonal complement, these subspaces can be parametrized with a single vector of $\mathbb{R}^k$. The next result shows that $\mathbb{Z}^k$-actions of positive mean dimension must have infinite directional mean dimension with respect to each $\langle \vec{e_i}\rangle^{\perp}$ ($1\le i\le k$).

\begin{proposition}
If a $\mathbb{Z}^k$-action $(X,\mathbb{Z}^k,T)$ satisfies $\min_{1\le i\le k}\mdim(X,T,\langle\vec{e_i}\rangle^{\perp})<+\infty$ then $\mdim(X,\mathbb{Z}^k,T)=0$.
\end{proposition}

\begin{proof}
Without loss of generality we assume $\mdim(X,T,\langle\vec{e_k}\rangle^{\perp})<+\infty$. For any positive integers $N,M$ and $j$ we put $F(N,M;j)=[-jN,jN]^{k-1}\times[-jM,jM]$. We set $\mathcal{F}=\{F(N,M;1):N,M\ge1\}$. Clearly, $\{F(N,M;j)\}_{j=1}^\infty$ is a F\o lner sequence of $\mathbb{Z}^k$, for any positive integers $N$ and $M$; and $\mathcal{F}=\bigcup_{N,M\ge1}\{F(N,M;j)\}_{j=1}^\infty$. We rewrite the definition of $\mdim(X,T;\mathcal{U})$ using $\{F(N,M;j)\}_{j=1}^\infty$ in \eqref{def:mdim-degree}. By subadditivity of $\mathcal{D}$ we have for any finite open cover $\mathcal{U}$ of $X$ and any $N,M\ge1$ $$\mdim(X,T;\mathcal{U})=\lim_{j\to+\infty}\frac{\mathcal{D}(\bigvee_{n\in F(N,M;j)}T^n\mathcal{U})}{|F(N,M;j)|}=\inf_{j\ge1}\frac{\mathcal{D}(\bigvee_{n\in F(N,M;j)}T^n\mathcal{U})}{|F(N,M;j)|}.$$
Thus,
\begin{align*}
\mdim(X,T;\mathcal{U})
&=\inf_{N,M,j\ge1}\frac{\mathcal{D}(\bigvee_{n\in F(N,M;j)}T^n\mathcal{U})}{|F(N,M;j)|}=\inf_{F\in\mathcal{F}}\frac{\mathcal{D}(\bigvee_{n\in F}T^n\mathcal{U})}{|F|}\\
&=\inf_{N,M\ge1}\frac{1}{2M+1}\cdot\frac{\mathcal{D}(\bigvee_{n\in[-N,N]^{k-1}}(T|_{\mathbb{Z}^k_k})^n(\bigvee_{m\in[-M,M]}T^{m\vec{e_k}}\mathcal{U}))}{(2N+1)^{k-1}}\\
&=\inf_{M\ge1}\frac{1}{2M+1}\cdot\inf_{N\ge1}\frac{\mathcal{D}(\bigvee_{n\in[-N,N]^{k-1}}(T|_{\mathbb{Z}^k_k})^n(\bigvee_{m\in[-M,M]}T^{m\vec{e_k}}\mathcal{U}))}{(2N+1)^{k-1}}.
\end{align*}
Since $\bigvee_{m\in[-M,M]}T^{m\vec{e_k}}\mathcal{U}$ is independent of $N$, we have $$\mdim(X,T;\mathcal{U})\le\inf_{M\ge1}\frac{\mdim(X,\mathbb{Z}^{k-1},T|_{\mathbb{Z}^k_k})}{2M+1}=\inf_{M\ge1}\frac{\mdim(X,T,\langle\vec{e_k}\rangle^{\perp})}{2M+1}.$$ Since $\mdim(X,T,\langle\vec{e_k}\rangle^{\perp})$ is finite and since $\mathcal{U}$ is an arbitrary finite open cover of $X$, we conclude that $\mdim(X,\mathbb{Z}^k,T)=0$.
\end{proof}
We note that for a $\mathbb{Z}^k$-action it is possible for us to consider its $(k-1)$-dimensional directional mean dimension as the $[0,+\infty]$-valued function on the $(k-1)$-sphere $\mathbb{S}^{k-1}=\{\vec{v} \in \mathbb{R}^{k-1}: \|\vec{v}\|=1\}$ by associating to $\vec{v}\in\mathbb{S}^{k-1}$ the value of $\mdim(X,T,\langle\vec{v}\rangle^{\perp})$. The next theorem reveals this function is \textit{not} necessarily continuous on $\mathbb{S}^{k-1}$ even for $k=2$.

\begin{theorem}[=Theorem \ref{thm:discontinuity}]
For any $\rho\in[0,+\infty]$ there is a $\mathbb{Z}^2$-action $(X,\mathbb{Z}^2,T)$ such that if $\vec{v}\in\{(-1,0),(1,0)\}$ then $\mdim(X,T,\langle \vec{v}\rangle^{\perp} )=\rho$; if $\vec{v}\in\mathbb{S}\setminus\{(-1,0),(1,0)\}$ then $\mdim(X,T,\langle\vec{v}\rangle^{\perp})=0$. In particular, if $\rho>0$, the $(k-1)$-dimensional directional mean dimension of $(X,\mathbb{Z}^2,T)$ is not continuous on $\mathbb{S}$.

\end{theorem}
\begin{proof}
We provide a constructive proof, following the idea of Example 6.6 in \cite{Boyle--Lind}. The case $\rho=0$ is trivial. We can suppose $\rho=1$. Indeed, we may replace (in the following argument) $[0,1]^\mathbb{Z}$ by $([0,1]^\mathbb{N})^\mathbb{Z}$ if $\rho=+\infty$; or by a subshift (even minimal \cite{Lindenstrauss--Weiss}) of some $([0,1]^d)^\mathbb{Z}$ (where $d>\rho$), whose mean dimension is equal to $\rho$, if $0<\rho<+\infty$. So we assume $\rho=1$ for simplicity.

We let $\sigma$ be the full shift on $[0,1]^\mathbb{Z}$, $\sigma\colon \mathbb{Z}\times[0,1]^\mathbb{Z}\to[0,1]^\mathbb{Z}$, $(n,(x_m)_{m\in\mathbb{Z}})\mapsto(x_{m+n})_{m\in\mathbb{Z}}$. We set $Y=[0,1]^\mathbb{Z}\times\mathbb{Z}$. We consider an action on $Y$ as follows:
\begin{equation}\label{equ:example1}
\mathbb{Z}^2\times Y\to Y,\quad(m,n)(x,i)=(\sigma^mx,n+i),\quad\forall m,n,i\in\mathbb{Z},\,\forall x\in[0,1]^\mathbb{Z}.
\end{equation}
Let $X=Y\cup\{\infty\}$ be the one-point compactification of $Y$. By defining $(m,n)\infty=\infty$ for all $(m,n)\in\mathbb{Z}^2$ (i.e. $\infty$ is a fixed point) we may extend the action \eqref{equ:example1} of $\mathbb{Z}^2$ on $Y$ continuously to a $\mathbb{Z}^2$-action $T$ on $X$.

In the remaining part of the proof we shall compute the directional mean dimension of $(X,\mathbb{Z}^2,T)$. We fix an arbitrary $\vec{v}=(v_1,v_2)\in\mathbb{S}$.
\medskip

\textbf{Case 1.}
Suppose $v_2\ne0$. We are going to prove $\mdim(X,T,\langle\vec{v}\rangle^{\perp})=0$.

In fact, for any finite open cover $\mathcal{U}$ of $X$ there exist $M\in\mathbb{N}$ sufficiently large and a finite open cover $\alpha$ of $[0,1]^\mathbb{Z}$ such that $\mathcal{V}$ refines $\mathcal{U}$, where $\mathcal{V}$ is a finite open cover of $X$ consisting of the following members:
\begin{align*}
&\bullet\quad\quad
A\times\{i\}\subset X,\quad\forall A\in\alpha,\,\forall i\in[-M,M],\\
&\bullet\quad\quad
[0,1]^\mathbb{Z}\times\left(\mathbb{Z}\setminus[-M,M]\right)\cup\{\infty\}\subset X.
\end{align*}
Hence $T^{(m,n)}\mathcal{V}$ ($(m,n)\in\mathbb{Z}^2$) consists of the following members:
\begin{align*}
&\bullet\quad\quad
\sigma^m(A)\times\{i+n\},\quad\forall A\in\alpha,\,\forall i\in[-M,M],\\
&\bullet\quad\quad
[0,1]^\mathbb{Z}\times\left(\mathbb{Z}\setminus[n-M,n+M]\right)\cup\{\infty\}.
\end{align*}
We thus observe for any $(m,n)\in\mathbb{Z}^2$ with $n\ge2M+1$ that $\mathcal{V}\vee T^{(m,n)}\mathcal{V}$ consists of the following members:
\begin{align*}
&\bullet\quad\quad
[0,1]^\mathbb{Z}\times\left(\mathbb{Z}\setminus([-M,M]\cup[n-M,n+M])\right)\cup\{\infty\},\\
&\bullet\quad\quad
A\times\{i\},\quad\forall A\in\alpha,\,\forall i\in[-M,M],\\
&\bullet\quad\quad
\sigma^m(A)\times\{i+n\},\quad\forall A\in\alpha,\,\forall i\in[-M,M].
\end{align*}
We note that any two elements coming from distinct lines (in the above three) are disjoint. So we have $\mathcal{D}(\mathcal{V})\le\mathcal{D}(\mathcal{V}\vee T^{(m,n)}\mathcal{V})\le\max\{\mathcal{D}(\mathcal{V}),\mathcal{D}(T^{(m,n)}\mathcal{V})\}=\mathcal{D}(\mathcal{V})$. Similarly, for any $(m,n)\in\mathbb{Z}^2$ with $n\le-2M-1$ we also have $\mathcal{D}(\mathcal{V}\vee T^{(m,n)}\mathcal{V})=\mathcal{D}(\mathcal{V})$. Therefore, by induction we deduce that if a finite subset $F$ of $\mathbb{Z}^2$ satisfies that $|q_1-q_2|\ge2M+1$ whenever $(p_1,q_1)\ne(p_2,q_2)\in F$, then $\mathcal{D}(\bigvee_{(p,q)\in F}T^{(p,q)}\mathcal{V})=\mathcal{D}(\mathcal{V})$. Thus for any $N\in\mathbb{N}$ we have $$\mathcal{D}(\bigvee_{(m,n)\in B_1(\langle\vec{v}\rangle^{\perp})\cap[-N,N]^2}T^{(m,n)}\mathcal{V})\le\left|B_1(\langle\vec{v}\rangle^{\perp})\cap[-M-1,M+1]^2\right|\cdot\mathcal{D}(\mathcal{V}).$$ 
Notice that we are using $r=1> \sqrt{2}/2$ in the definition of directional mean dimension. It follows that,  
\begin{align*}
\mdim(X,T,\langle\vec{v}\rangle^{\perp})
&=\sup_\mathcal{U}\liminf_{N\to+\infty}\frac{\mathcal{D}(\bigvee_{(m,n)\in B_1(\langle\vec{v}\rangle^{\perp})\cap[-N,N]^2}T^{(m,n)}\mathcal{U})}{\vol_1(\langle\vec{v}\rangle^{\perp}\cap[-N,N]^2)}\\
&\le\sup_\mathcal{U}\liminf_{N\to+\infty}\frac{\mathcal{D}(\bigvee_{(m,n)\in B_1(\langle\vec{v}\rangle^{\perp})\cap[-N,N]^2}T^{(m,n)}\mathcal{V})}{\vol_1(\langle\vec{v}\rangle^{\perp}\cap[-N,N]^2)}\\
&\le\sup_\mathcal{U}\liminf_{N\to+\infty}\frac{\left|B_1(\langle\vec{v}\rangle^{\perp})\cap[-M-1,M+1]^2\right|\cdot\mathcal{D}(\mathcal{V})}{\vol_1(\langle\vec{v}\rangle^{\perp}\cap[-N,N]^2)}\\
&=0.
\end{align*}

\textbf{Case 2.}
Suppose $v_2=0$. We are going to show $\mdim(X,T,\langle\vec{v}\rangle^{\perp})=1$.

In fact, by Proposition \ref{prop:ssss}.$(3)$,  we have

$$\mdim(X,T,\langle\vec{v}\rangle^{\perp})=\sup_{\alpha,M}\liminf_{N\to+\infty}\frac{\mathcal{D}(\bigvee_{n\in[-N,N]}T^{(n,0)}\mathcal{U}_{\alpha,M})}{2N+1},$$where $\mathcal{U}_{\alpha,M}$ is the finite open cover of $X$ consisting of the following members:
\begin{align*}
&\bullet\quad\quad
A\times\{i\},\quad\forall A\in\alpha,\,\forall i\in[-M,M],\\
&\bullet\quad\quad
[0,1]^\mathbb{Z}\times\left(\mathbb{Z}\setminus[-M,M]\right)\cup\{\infty\},
\end{align*}
and where $\alpha$ ranges over finite open covers of $[0,1]^\mathbb{Z}$ and $M$ runs over $\mathbb{N}$. Meanwhile, it is clear that $\mathcal{D}(\bigvee_{n\in[-N,N]}T^{(n,0)}\mathcal{U}_{\alpha,M})=\mathcal{D}(\bigvee_{n\in[-N,N]}\sigma^n\alpha)$. Thus we deduce
$$\mdim(X,T,\langle\vec{v}\rangle^{\perp})=\sup_\alpha\liminf_{N\to+\infty}\frac{\mathcal{D}(\bigvee_{n\in[-N,N]}\sigma^n\alpha)}{2N+1}=\mdim([0,1]^\mathbb{Z},\sigma)=1.$$\end{proof}

\section{Proof of Theorem \ref{maintheorem2}}
\begin{lemma}[{\cite[Theorems 1.25, 1.26]{Nadler2}}]\label{lem:continuumarc}
For any compact metric space $X$, $A\in C(X)$ and $a\in A$ there is a continuous mapping $\kappa:[0,1]\to C(X)$ such that $\kappa(0)=\{a\}$, $\kappa(1)=A$, and $\kappa(x)\subset\kappa(y)$ for all $0\le x\le y\le1$.
\end{lemma}

\begin{lemma}\label{lem:decreasingclopen}
Let $X$ be a compact metric space and $C$ a connected component of $X$. Then there exists a sequence $\{F_n\}_{n=1}^\infty$ of clopen subsets of $X$ satisfying $F_1\supset F_2\supset\cdots\supset F_n\supset F_{n+1}\supset\cdots\supset\bigcap_{n=1}^\infty F_n=C$.
\end{lemma}
\begin{proof}
Without loss of generality we assume $C\ne X$. Fix $x\in C$. We denote by $\mathcal{Q}_x$ the set of all clopen neighbourhoods of $x$ in $X$ and let $Q=\bigcap_{V\in\mathcal{Q}_x}V$ be the {\em quasi-component} of $x$. It follows from classical results that $C=Q$ (see for instance \cite[\S47, Section II, Theorem 2]{Kuratowski}; we thank the anonymous referee for pointing us to this reference)·

Since $C=Q$, for every $a\in X\setminus C$ there must be a clopen subset $C_a$ of $X$ such that $C\subset C_a$ and $a\notin C_a$. This implies that there is an open neighbourhood $U_a$ of $a$ in $X$ with $U_a\subset X\setminus C_a\subset X\setminus C$. We notice that $X\setminus C$ is a Lindel\"of space and that $\{U_a:a\in X\setminus C\}$ is an open cover of $X\setminus C$. So we can find $\{a_i\}_{i=1}^\infty\subset X\setminus C$ such that $X\setminus C=\bigcup_{i=1}^\infty U_{a_i}$. Let $F=\bigcap_{i=1}^\infty C_{a_i}$. Since for every $c\in X\setminus C$ there is some $i\in\mathbb{N}$ with $c\in U_{a_i}\subset X\setminus C_{a_i}\subset X\setminus F$, we see $F\subset C$. Thus, $F=C$. For each $n\in\mathbb{N}$ we set $F_n=\bigcap_{i=1}^nC_{a_i}$. It is simple to check that $\{F_n\}_{n=1}^\infty$ is as required.
\end{proof}
Let $(X,\mathbb{Z}^k,T)$ be a continuum-wise expansive $\mathbb{Z}^k$-action with an expansivity constant $2c>0$.
\begin{lemma}\label{prop:contexpansive}
(1) There is $\delta>0$ such that if $A\in C(X)$ and $N\ge1$ satisfy the condition $c\le\max\{\diam T^n(A):n\in[-N,N]^k\}\le2c$ then $\max\{\diam T^n(A):n\in\partial[-N,N]^k\}>\delta$.

(2) For any $\epsilon>0$ there is $m=m(\epsilon)>0$ such that if $A\in C(X)$ satisfies the condition $\max\{\diam T^n(A):n\in[-m,m]^k\}\le2c$ then $\diam A<\epsilon$.
\end{lemma}
\begin{proof}
(1) Suppose that the statement is false. There exist $\{A_i\}_{i\in\mathbb{N}}\subset C(X)$ and $\{N_i\}_{i\in\mathbb{N}}\subset\mathbb{N}$ satisfying that
\begin{equation}\label{equ:diampart}
c\le\max\{\diam T^n(A_i):n\in[-N_i,N_i]^k\}\le2c,\quad\forall i\in\mathbb{N};
\end{equation}
\begin{equation}\label{equ:diampartial}
\max\{\diam T^n(A_i):n\in\partial[-N_i,N_i]^k\}\le1/i,\quad\forall i\in\mathbb{N}.
\end{equation}
Suppose that the sequence $\{N_i\}_{i\in\mathbb{N}}\subset\mathbb{N}$ is bounded by some $0<M<+\infty$. We may assume (by the compactness of $C(X)$) that $\{A_i\}_{i\in\mathbb{N}}$ converges to some $A\in C(X)$. Hence \eqref{equ:diampart} and \eqref{equ:diampartial} lead to a contradiction for $\diam T^n(A)$ for some $n\in[-M,M]^k$. So we assume without loss of generality (by taking a subsequence) that $N_i\to+\infty$ as $i\to\infty$.

We take a sequence $\{n_i\}_{i\in\mathbb{N}}$, with $n_i\in [-N_i,N_i]^k$ and such that $\diam T^{n_i}(A_i) =  \max \{\diam  T^n(A_i):n\in[-N_i,N_i]^k\}$ for all $i\in\mathbb{N}$. We claim that the distance between $n_i$ and $\partial[-N_i,N_i]^k$ goes to $+\infty$ as $i\to\infty$. In fact, if this assertion is not true, then there exists $l\in\mathbb{N}$ such that $n_i\notin[-N_i+l,N_i-l]^k$ for infinitely many $i\in\mathbb{N}$. For all these $i\in\mathbb{N}$ we write $n_i=p_i+q_i$ with $p_i\in[-l,l]^k$ and $q_i\in\partial[-N_i,N_i]^k$. We take $\epsilon>0$ sufficiently small such that $\diam T^n(A)<c/2$ for all $n\in[-l,l]^k$ whenever $\diam A<\epsilon$. It follows from \eqref{equ:diampartial} that $\diam T^{n_i}(A_i)=\diam T^{p_i}(T^{q_i}(A_i))<c/2$ for all sufficiently large $i$ (with $1/i<\epsilon$), which contradicts \eqref{equ:diampart}, and thus proves the claim.

We assume (by choosing a subsequence) that $T^{n_i}(A_i)$ tends to $B\in C(X)$ as $i\to\infty$. For any given $m\in\mathbb{Z}^k$ there is some $i_0\in\mathbb{N}$ such that $m+n_i\in[-N_i,N_i]^k$, for all $i\ge i_0$, which, together with \eqref{equ:diampart}, implies $\diam T^{m+n_i}(A_i)\le\max\{\diam T^n(A_i):n\in[-N_i,N_i]^k\}\le2c$. Letting $i\to\infty$ in the above inequality we obtain $\diam T^m(B)\le2c$. Since $B$ is nondegenerate, this is a contradiction.

(2) Suppose that the statement is false. There exist $\epsilon_0>0$ and a sequence $\{A_i\}_{i\in\mathbb{N}}\subset C(X)$ such that $\max\{\diam T^n(A_i):n\in[-i,i]^k\}\le2c$ and $\diam A_i\ge\epsilon_0$ for all $i\in\mathbb{N}$. Choose a subsequence of $\{A_i\}_{i\in\mathbb{N}}$ tending to some $A\in C(X)$. Clearly, $\diam T^n(A)\le2c$ for all $n\in\mathbb{Z}^k$, and $\diam A\ge\epsilon_0$. This is a contradiction.
\end{proof}
\begin{lemma}\label{prop:bounded}
There is a constant $K>0$ and a sequence $\{\mathcal{W}_N\}_{N\ge1}$ of finite closed covers of $X$ satisfying that $\mesh(\mathcal{W}_N,d^T_{[-N,N]^k})<2c$ and $\ord(\mathcal{W}_N)\le KN^{k-1}$, for all $N\ge1$.
\end{lemma}
\begin{proof}
Let $\delta>0$ be a constant introduced in Lemma \ref{prop:contexpansive} (1). Choose a finite open cover $\mathcal{U}=\{U_1,\dots,U_L\}$ of $X$ with $\mesh(\mathcal{U},d)<\delta/2$. For each $N\ge1$ we let $\mathcal{U}_N$ be an open cover of $X$ defined by $\mathcal{U}_N=\bigvee_{u\in\partial[-N,N]^k}T^{-u}\mathcal{U}$. By Lemma \ref{lem:subadditivedegree}
\begin{equation}\label{equ:boundedorder}
\mathcal{D}(\mathcal{U}_N)\le|\partial[-N,N]^k|\cdot\mathcal{D}(\mathcal{U})\le2^k(2N+1)^{k-1}L,\quad\forall N\ge1.
\end{equation}
By Lemma \ref{lem:orderclosedcover}, for each $N\ge1$ there is a finite closed cover $\mathcal{V}_N$ of $X$, which refines $\mathcal{U}_N$ and which satisfies $\ord(\mathcal{V}_N)=\mathcal{D}(\mathcal{U}_N)$. For every $V\in\mathcal{V}_N$ we denote by $\mathcal{C}_V$ the set of all connected components of $V$. We set $\mathcal{C}_N=\{C\subset V:V\in\mathcal{V}_N,C\in\mathcal{C}_V\}$ and note that the sets in $\mathcal{C}_V$ are pairwise disjoint. The set $\mathcal{C}_N$ is a closed (not necessarily finite) cover of $X$, which refines $\mathcal{V}_N$ and satisfies $\ord(\mathcal{C}_N)=\ord(\mathcal{V}_N)$.
\newline\textsf{Claim 1.}
For every $N\ge1$ we have $\mesh(\mathcal{C}_N,d^T_{[-N,N]^k})<2c$.
\newline\textit{Proof of Claim 1.}
If $\mesh(\mathcal{C}_N,d^T_{[-N,N]^k})\ge2c$ for some $N\ge1$, then there exists $C\in\mathcal{C}_N$ such that $\max\{\diam T^n(C):n\in[-N,N]^k\}\ge2c$ which implies that $d^T_{[-N,N]^k}(x,y)\ge2c$ for some $x,y\in C$. We assume $C\in\mathcal{C}_V$ for some $V\in\mathcal{V}_N$. By Lemma \ref{lem:continuumarc}, there is a continuous mapping $\kappa\colon [0,1]\to C(V)$ satisfying $\kappa(0)=\{x\}$, $\kappa(1)=C$, and $\kappa(s)\subset\kappa(t)$ whenever $0\le s\le t\le1$. We define a continuous mapping $F\colon [0,1]\to[0,+\infty)$ by sending $t$ to $\max\{\diam T^n(\kappa(t)):n\in[-N,N]^k\}$. Clearly, $F(0)=0$ and $F(1)\ge2c$. So there is some $0<t_0<1$ with $F(t_0)=3c/2$. By Lemma \ref{prop:contexpansive} (1) we get $\max\{\diam T^n(\kappa(t_0)):n\in\partial[-N,N]^k\}>\delta$ and hence $\max\{\diam T^n(C):n\in\partial[-N,N]^k\}>\delta$. However, since $\mathcal{C}_N$ refines $\mathcal{U}_N$ and since $\mesh(\mathcal{U}_N,d_{\partial[-N,N]^k}^T)<\delta/2$, we have $\max\{\diam T^n(C):n\in\partial[-N,N]^k\}<\delta/2$, a contradiction. This shows Claim 1.
\newline\textsf{Claim 2.}
For every $N\ge1$ there exists a finite closed cover $\mathcal{W}_N$ of $X$ satisfying that $\mesh(\mathcal{W}_N,d^T_{[-N,N]^k})<2c$ and $\ord(\mathcal{W}_N)=\ord(\mathcal{V}_N)$.
\newline\textit{Proof of Claim 2.}
We fix $N\ge1$. We take an arbitrary $V\in\mathcal{V}_N$ for the moment. By Claim 1, $\mesh(\mathcal{C}_V,d^T_{[-N,N]^k})<2c$. Applying Lemma \ref{lem:decreasingclopen} to each connected component $C$ of $V$ we find a clopen subset $F_C\supset C$ of $V$, whose diameter is strictly smaller than $2c$ with respect to $d^T_{[-N,N]^k}$. So we obtain a clopen cover $\{F_C:C\in\mathcal{C}_V\}$ of $V$. Since $V$ is compact, there is a finite subcover $\{F_{C_1},F_{C_2},\dots,F_{C_{i(V)}}\}$ of $\{F_C:C\in\mathcal{C}_V\}$. We set $W_1^V=F_{C_1}$, $W_2^V=F_{C_2}\setminus F_{C_1}$, $\dots$, $W_{i(V)}^V=F_{C_{i(V)}}\setminus\cup_{j=1}^{i(V)-1}F_{C_j}$. It follows that $\{W_1^V,W_2^V,\dots,W_{i(V)}^V\}$ is a finite closed cover of $V$ satisfying that $W_s^V\cap W_t^V=\emptyset$ for all $1\le s<t\le i(V)$. Thus, we know that $\mathcal{W}_N=\{W_j^V:1\le j\le i(V),V\in\mathcal{V}_N\}$ is a finite closed cover of $X$. This shows Claim 2.
\newline Now by \eqref{equ:boundedorder} and noting that $\ord(\mathcal{W}_N)\le\mathcal{D}(\mathcal{U}_N)$ we finish the proof of the lemma.
\end{proof}

\begin{proof}[Proof of Theorem \ref{maintheorem2}]
With the above preparation we are now ready to prove Theorem \ref{maintheorem2}. We fix a nonzero vector $\vec{v}\in\mathbb{R}^k$ and $r>\frac{\sqrt{k}}{2}$. We take $\epsilon^\prime>0$ arbitrarily. We choose a finite open cover $\alpha$ of $X$, which consists of open balls of diameter smaller than $\epsilon^\prime$ and which has a Lebesgue number $\epsilon>0$, with respect to the distance $d$ on $X$. Note that we ensure $\epsilon<\epsilon^\prime$. Applying Lemma \ref{prop:contexpansive} (2) to $\epsilon>0$, there exists $m=m(\epsilon)>0$ such that for $A\in C(X)$, if $\max\{\diam T^{u+u^\prime}(A):u\in B_r(\langle\vec{v}\rangle^{\perp})\cap[-N,N]^k,u^\prime\in[-m,m]^k\}\le2c$ then $\max\{\diam T^u(A):u\in B_r(\langle\vec{v}\rangle^{\perp})\cap[-N,N]^k\}<\epsilon$. By Lemma \ref{prop:bounded}, there is some constant $K>0$ such that for every $N\in\mathbb{N}$ we can find a finite closed cover $\mathcal{W}_{N+m}=\{W_i\}_{i\in I_N}$ of $X$ satisfying $\mesh(\mathcal{W}_{N+m},d^T_{[-N-m,N+m]^k})<2c$ and $\ord(\mathcal{W}_{N+m})<K\cdot(N+m)^{k-1}$. For each $i\in I_N$ we denote by $\mathcal{P}_i$ the set of all connected components of $W_i$. We put $\mathcal{C}_N=\bigcup_{i\in I_N}\mathcal{P}_i$. We notice that connected subsets in $W_i$ ($i\in I_N$) are also connected in $X$. It follows that $\ord(\mathcal{C}_N)=\ord(\mathcal{W}_{N+m})$ and $\mesh(\mathcal{C}_N,d_{B_r(\langle\vec{v}\rangle^{\perp})\cap[-N,N]^k}^T)<\epsilon$. By a similar argument as in Claim 2 (in Lemma \ref{prop:bounded} above) we can find a finite closed cover $\mathcal{C}^\prime_N$ of $X$ satisfying that $\ord(\mathcal{C}^\prime_N)=\ord(\mathcal{C}_N)$ and $\mesh(\mathcal{C}^\prime_N,d_{B_r(\langle\vec{v}\rangle^{\perp})\cap[-N,N]^k}^T)<\epsilon$. Clearly, $\mathcal{C}_N^\prime$ refines $\bigvee_{n\in B_r(\langle\vec{v}\rangle^{\perp})\cap[-N,N]^k}T^n\alpha$. It follows from Lemma \ref{lem:orderclosedcover} that $\mathcal{D}(\bigvee_{n\in B_r(\langle\vec{v}\rangle^{\perp})\cap[-N,N]^k}T^n\alpha)\le\ord(\mathcal{C}^\prime_N)<K\cdot(N+m)^{k-1}$ and thus 
 $$\Widim_{\epsilon^\prime}(X,d_{B_r(\langle\vec{v}\rangle^{\perp})\cap[-N,N]^k}^T)\le\ord(\mathcal{C}^\prime_N)<K\cdot(N+m)^{k-1}.$$ 

We remark that the constant $K$ depends only on $(X,\mathbb{Z}^k,T)$. Since $m$ is independent of $N$ and since $\epsilon^\prime>0$ is arbitrary, the definition \eqref{equ:definitiondirectmdim} allows us to conclude.
\end{proof}

To end this paper we remark that we also have the following result which unifies Meyerovitch and Tsukamoto's theorem and Kato's theorem and which does \textit{not} include Theorem \ref{maintheorem2}. We omit its proof because it is highly similar to Theorem \ref{maintheorem2} and the previously-known results.
\begin{itemize}\item
Let $k$ be a positive integer. If $(X,\mathbb{Z}^k,T)$ is a continuum-wise expansive $\mathbb{Z}^k$-action and if a $\mathbb{Z}^{k-1}$-action $(X,\mathbb{Z}^{k-1},R)$ satisfies that $R\colon \mathbb{Z}^{k-1}\times X\to X$ commutes with $T\colon \mathbb{Z}^k\times X\to X$ \textup(namely $T^m\circ R^n(x)=R^n\circ T^m(x)$ for all $m\in\mathbb{Z}^k$, $n\in\mathbb{Z}^{k-1}$ and $x\in X$\textup), then $(X,\mathbb{Z}^{k-1},R)$ has finite mean dimension.
\end{itemize}



\end{document}